\documentclass[12pt,a4paper,reqno,twoside]{amsart}

\usepackage[english]{babel}
\usepackage{stmaryrd}
\usepackage{dsfont}
\usepackage[symbol*,ragged]{footmisc}
\usepackage[colorlinks,linkcolor=red,anchorcolor=blue,citecolor=blue,urlcolor=blue]{hyperref}
\usepackage{color,xcolor}

\usepackage{geometry}
\usepackage{amssymb}
\usepackage{amsmath}
\usepackage{mathrsfs}
\usepackage{amsfonts}
\usepackage{epsfig}

\usepackage{amsthm}
\usepackage{amsxtra}
\usepackage{bbding}
\usepackage{epsfig}
\usepackage{graphicx}
\usepackage{latexsym}
\usepackage{mathbbol}
\usepackage{bbold}

\usepackage{pifont}
\usepackage{wasysym}
\usepackage{skull}
\usepackage{float}

\DeclareSymbolFontAlphabet{\mathbb}{AMSb}
\DeclareSymbolFontAlphabet{\mathbbol}{bbold}

\usepackage{amscd}
\usepackage[all]{xy}
\allowdisplaybreaks[4]
\usepackage{setspace}

\theoremstyle{plain}
\newtheorem{theorem}{\normalfont\scshape Theorem}[section]
\newtheorem{proposition}{\normalfont\scshape Proposition}[section]
\newtheorem{lemma}[proposition]{\normalfont\scshape Lemma}

\newtheorem*{corollary*}{\normalfont\scshape Corollary}

\theoremstyle{remark}
\newtheorem*{remark*}{\normalfont\scshape Remark}
\newtheorem*{notation}{\normalfont\scshape Notation}

\numberwithin{equation}{section}
\addtocounter{footnote}{1}

\renewcommand{\footnoterule}{
  \kern -3pt
  \hrule width 2.5in height 0.4pt
  \kern 3pt
}

\makeatletter
\@ifundefined{MakeUppercase}{}{}
\makeatother

\begin{document}
	
\title[A quaternary Diophantine inequality with prime numbers of a special form (II) ]
	  {A quaternary Diophantine inequality with prime numbers of a special form (II) }

\author[Yuhua Zhao, Jinjiang Li, Linji Long, Min Zhang]
       { Yuhua Zhao \quad \& \quad Jinjiang Li \quad \& \quad Linji Long \quad \& \quad Min Zhang }

\address{[Yuhua Zhao] Department of Mathematics, China University of Mining and Technology,
         Beijing 100083, People's Republic of China}
	
\email{\textcolor{black}{yuhua.zhao.math@gmail.com}}
	
\address{[Jinjiang Li] (Corresponding author) Department of Mathematics, China University of Mining and Technology,
         Beijing 100083, People's Republic of China}
	
\email{\textcolor{black}{jinjiang.li.math@gmail.com}}

\address{[Linji Long] Department of Mathematics, China University of Mining and Technology,
         Beijing 100083, People's Republic of China}

\email{\textcolor{black}{linji.long.math@gmail.com}}

\address{[Min Zhang] School of Applied Science, Beijing Information Science and Technology University,
		 Beijing 100192, People's Republic of China  }
	
\email{\textcolor{black}{min.zhang.math@gmail.com}}

\date{}

\footnotetext[1]{Jinjiang Li is the corresponding author. 

\quad\,\,
{\textbf{MR(2020) Subject Classification}}: 11L07, 11L20, 11P05, 11P32, 11N36

}

\begin{abstract}
Let $N$ be a sufficiently large real number. In this paper, it is proved that, for $1<c<\frac{5179}{3909}$ and for any arbitrary large number $E>0$, the Diophantine inequality
\begin{equation*}
   \big|p_1^c+p_2^c+p_3^c+p_4^c-N\big|<(\log N)^{-E}
\end{equation*}
is solvable in prime variables $p_1,p_2,p_3,p_4$ such that each of the numbers $p_i+2\,(i=1,2,3,4)$ has, counted according to multiplicity, at most $[\frac{31250}{5179-3909c}]$ prime factors, which constitutes a large improvement upon the previous results of Dimitrov \cite{Dimitrov-2017-3} and Li--Xue--Zhang \cite{Li-Xue-Zhang-2024}, respectively.
\end{abstract}

\maketitle

\section{Introduction and main result}
For fixed integer $k\geqslant1$ and sufficiently large integer $N$, the well--known Waring--Goldbach problem is devoted to investigating the solvability of the following Diophantine equality
\begin{equation}\label{com-WG}
  N=p_1^k+p_2^k+\cdots+p_s^k
\end{equation}
in prime variables $p_1,p_2,\dots,p_s$. In this topic, many mathematicians have derived many splendid results. For instance, in 1937, Vinogradov \cite{Vinogradov-1937} proved that such a representation of the type (\ref{com-WG}) exists for every sufficiently large odd integer with $k=1,s=3$. Moreover, in 1938, Hua \cite{Hua-1938} showed that (\ref{com-WG}) is solvable
for every sufficiently large integer $N$ satisfying $N\equiv 5\!\pmod {24}$ with $k=2,s=5$.

In 1952, Piatetski--Shapiro \cite{Piatetski-Shapiro-1952} studied the following analog of the Waring--Goldbach problem. Suppose that $c>1$ is not an integer and $\varepsilon$ is a small positive number. Denote by $H(c)$ the smallest natural number $s$ such that, for every
sufficiently large real number $N$, the Diophantine inequality
\begin{equation}\label{WG-analogue}
  \big|p_1^c+p_2^c+\cdots+p_s^c-N\big|<\varepsilon
\end{equation}
is solvable in primes $p_1,p_2,\dots,p_s$. Then it was proved in \cite{Piatetski-Shapiro-1952} that
\begin{equation*}
 \limsup_{c\to+\infty}\frac{H(c)}{c\log c}\leqslant 4.
\end{equation*}
Also, in \cite{Piatetski-Shapiro-1952}, Piatetski--Shapiro considered the case $s=5$ in (\ref{WG-analogue}) and proved that
$H(c)\leqslant5$ for $1<c<3/2$. Later, the upper bound $3/2$ for $H(c)\leqslant5$ was improved
successively to
\begin{equation*}
  \frac{14142}{8923},\quad \frac{1+\sqrt{5}}{2},\quad \frac{81}{40},\quad \frac{108}{53}, \quad 2.041, \quad \frac{52}{25},\quad\frac{378}{181}
\end{equation*}
by Zhai and Cao \cite{Zhai-Cao-2003}, Garaev \cite{Garaev-2003}, Zhai and Cao \cite{Zhai-Cao-2007}, Shi and Liu \cite{Shi-Liu-2013}, Baker and Weingartner \cite{Baker-Weingartner-2013},  Li and Cai \cite{Li-Cai-2020}, Baker \cite{Baker-2021}, respectively.

From these results and the Goldbach--Vinogradov theorem, it is reasonable to conjecture that if $c$ is near to $1$, then the Diophantine inequality (\ref{WG-analogue}) is solvable for $s=3$. This conjecture was first established by Tolev \cite{Tolev-thesis} for $1<c<\frac{27}{26}$. Since then, the range of $c$ was enlarged to
\begin{equation*}
\frac{15}{14},\quad   \frac{13}{12},\quad   \frac{11}{10},\quad   \frac{237}{214},\quad   \frac{61}{55},\quad   \frac{10}{9},\quad   \frac{43}{36}
\end{equation*}
by Tolev \cite{Tolev-1992}, Cai \cite{Cai-1996}, Cai \cite{Cai-1999} and Kumchev and Nedeva
\cite{Kumchev-Nedeva-1998} independently, Cao and Zhai \cite{Cao-Zhai-2002}, Kumchev \cite{Kumchev-1999}, Baker and Weingartner \cite{Baker-Weingartner-2014}, Cai \cite{Cai-2018}, successively and respectively. The hitherto best result up to now belongs to Baker \cite{Baker-2021} with $1<c<6/5$.

Combining Tolev's method and the techniques of estimates on exponential sums of Fouvry and Iwaniec, in 2003, Zhai and Cao \cite{Zhai-Cao-2003} proved firstly that $H(c)\leqslant4$ for $1<c<\frac{81}{68}$. Later, the range of $c$ for
$H(c)\leqslant4$ was enlarged to
\begin{equation*}
 \frac{97}{81}, \qquad \frac{6}{5},\qquad\frac{59}{44},\qquad\frac{1193}{889},\qquad \frac{39}{29}
\end{equation*}
by Mu \cite{Mu-2015}, Zhang and Li \cite{Zhang-Li-2019-1}, Li and Cai \cite{Li-Cai-2020}, Zhang and Li \cite{Zhang-Li-2019-2}, Baker \cite{Baker-2021}, respectively.

Another central problem in the theory of prime distribution, namely the twin prime conjecture, states that there exist infinitely many primes $p$ such that $p+2$ is also prime. Although this conjecture has resisted all attacks, there have been spectacular partial achievements. Let $\mathcal{P}_r$ denote an almost--prime with at most $r$ prime factors, counted according to multiplicity. One well--known result is due to Chen \cite{Chen-1966,Chen-1973}, who proved that there exist infinitely many primes $p$ such that $p+2$ has at most $2$ prime factors.

Bearing in mind the result of Chen \cite{Chen-1966,Chen-1973}, one may try to study the arithmetical properties of the set of primes $p$ such that $p+2\in\mathcal{P}_r$ for a fixed $r\geqslant2$ and, in particular, to establish the solvability of Diophantine equations or inequalities in such primes. For instance, combining the results of Vinogradov \cite{Vinogradov-1937} and Chen \cite{Chen-1973}, Tolev \cite{Tolev-1999,Tolev-2000-1,Tolev-2000-2} established such kinds of results, while Matom\"{a}ki and Shao \cite{Matomaki-Shao-2017} improved the result of Tolev \cite{Tolev-2000-2} and proved that every sufficiently large odd integer $N$ can be represented as a sum of three primes $p_1,p_2,p_3$ such that
$p_i+2\in\mathcal{P}_2\,(i=1,2,3)$.

Motivated by Tolev \cite{Tolev-thesis,Tolev-1992} and Chen \cite{Chen-1966,Chen-1973}, it is reasonable to conjecture that if the constant $c>1$ is close to one, then inequality (\ref{WG-analogue}), with a suitable $\varepsilon=\varepsilon(N)$ satisfying $\varepsilon(N)\to0$ as $N\to\infty$, is solvable in primes $p_i$ such that $p_i+2$ are almost primes of a certain fixed order for $s=3$. An attempt to establish a result of this type was first made by Dimitrov \cite{Dimitrov-2017-1}, he dealt with this problem with $0<c<4/21$ and $p_i+2=\mathcal{P}_{10}, i=1,2,3$. After that, the next step also belongs to Dimitrov \cite{Dimitrov-2017-2} with $1<c<121/120$ and $p_i+2=\mathcal{P}_{29}, i=1,2,3$. Later, motivated by Dimitrov \cite{Dimitrov-2017-1}, Tolev \cite{Tolev-2017} proved that, for $1<c<\frac{15}{14}$, the Diophantine inequality
\begin{equation}\label{case-3}
  \big|p_1^c+p_2^c+p_3^c-N\big|<(\log N)^{-E}
\end{equation}
is solvable in primes $p_1,p_2,p_3$ such that each of the numbers $p_i+2\,(i=1,2,3)$ has at most $[\frac{369}{180-168c}]$ prime factors, counted according to multiplicity, where $E>0$ is a sufficiently large constant. Recently, Zhu \cite{Zhu-2020} improved the result of Tolev \cite{Tolev-2017} and showed that for $1<c<\frac{281}{250}$, (\ref{case-3}) is solvable in primes $p_1,p_2,p_3$ such that each of the numbers $p_i+2\,(i=1,2,3)$ has at most $[\frac{1475}{562-500c}]$ prime factors, counted according to multiplicity.
Afterwards, Li, Xue and Zhang \cite{Li-Xue-Zhang-2022} continued improving the result of Zhu \cite{Zhu-2020} and
showed that for $1<c<\frac{973}{856}$, (\ref{case-3}) is solvable in primes $p_1,p_2,p_3$ such that each of the numbers $p_i+2\,(i=1,2,3)$ has at most $[\frac{12626}{4865-4280c}]$ prime factors, counted according to multiplicity.

For the case $s=4$ of this type, in 2017, Dimitrov \cite{Dimitrov-2017-3} proved that, for $1<c<\frac{832}{825}$,
the quaternary Diophantine inequality
\begin{equation}\label{case-4}
  \big|p_1^c+p_2^c+p_3^c+p_4^c-N\big|<(\log N)^{-E}
\end{equation}
is solvable in primes $p_1,p_2,p_3,p_4$ such that each of the numbers $p_i+2\,(i=1,2,3)$ has at most $32$ prime factors, counted according to multiplicity. Afterwards, Li--Xue--Zhang \cite{Li-Xue-Zhang-2024} improved the result of Dimitrov \cite{Dimitrov-2017-3} and showed that for $1<c<\frac{4803}{4040}$, (\ref{case-4}) is solvable in primes $p_1,p_2,p_3,p_4$ such that each of the numbers $p_i+2\,(i=1,2,3,4)$ has at most $[\frac{31540280}{12007500-10100000c}]$ prime factors, counted according to multiplicity.

In this paper, we shall continue to improve the result of Li--Xue--Zhang \cite{Li-Xue-Zhang-2024}, and establish the following theorem.

\begin{theorem}\label{Theorem-1}
  Suppose that $1<c<\frac{5179}{3909}$ and let $N$ be a sufficiently large real number. Then for any arbitrary
large number $E>0$, the Diophantine inequality
\begin{equation}\label{Thm-ineq}
   \big|p_1^c+p_2^c+p_3^c+p_4^c-N\big|<(\log N)^{-E}
\end{equation}
is solvable in prime variables $p_1,p_2,p_3,p_4$ such that each of the numbers $p_i+2\,(i=1,2,3,4)$ has at most
$\left[\frac{31250}{5179-3909c}\right]$ prime factors, counted according to multiplicity.
\end{theorem}

\begin{remark*}
In order to compare our result with the results of Dimitrov \cite{Dimitrov-2017-3} and Li--Xue--Zhang \cite{Li-Xue-Zhang-2024}, we list the numerical result as follows:
\begin{equation*}
 \frac{5179}{3909}=1.324891\dots;\qquad \frac{4803}{4040}=1.18886138\dots;\qquad \frac{832}{825}=1.0084848\dots.
\end{equation*}
\end{remark*}

\begin{notation}
In this paper, we denote by $\varepsilon$ and $A$ an arbitrarily small positive number and an arbitrarily large constant, respectively, which may not be the same in different formula. The letter
$p$, with or without subscript, always denotes a prime number. As usual, we use $d(n),\mu(n),\varphi(n),\Lambda(n)$ to denote Dirichlet's divisor function, M\"{o}bius' function, Euler's function and von Mangoldt's function, respectively. Moreover, we shall use $(m,n)$ and $[m,n]$ for the greatest common divisor and the least common multiple of the integers $m$ and $n$, respectively. Let $[t]$ be the integer part of the real number $t$
and $e(t)=\exp(2\pi it)$. $f(x)\ll g(x)$ means that $f(x)=O(g(x))$; $f(x)\asymp g(x)$ means that $f(x)\ll g(x)\ll f(x)$. Assume that $E>0$ is any arbitrary large number. In addition, we define
\begin{equation*}
  1<c<\frac{5179}{3909},\qquad X=N^{1/c},\qquad \delta=\frac{5179-3909c}{10000},\qquad \xi=\frac{5}{7}-\frac{21}{10}\delta,\qquad
  \eta=\frac{8}{25}\delta,
\end{equation*}
\begin{equation*}
  z=X^\eta,\qquad D=X^\delta,\qquad \tau=X^{\xi-c},\qquad P(z)=\prod_{2<p<z}p, \qquad \Xi=(\log X)^{E+3}.
\end{equation*}
\end{notation}


\section{Preliminary Lemmas}
In this section, we shall give some preliminary lemmas, which are necessary in the proof of Theorem \ref{Theorem-1}.

\begin{lemma}\label{xiaobei-lemma}
   Let $a$ and $b$ be real numbers with $0<b<a/4$, and let $r$ be a positive integer. Then there exists a function $\vartheta(y)$ which is $r$ times continuously differentiable and such that
  \begin{equation*}
     \begin{cases}
          \vartheta(y)=1,    &   \textrm{for \quad} |y|\leqslant a-b, \\
          0<\vartheta(y)<1,  &   \textrm{for \quad} a-b<|y|< a+b, \\
          \vartheta(y)=0,    &   \textrm{for \quad} |y|\geqslant a+b,
      \end{cases}
  \end{equation*}
  and its Fourier transform
   \begin{equation*}
      \Theta(x)=\int_{-\infty}^{+\infty} \vartheta(y)e(-xy)\mathrm{d}y
   \end{equation*}
   satisfies the inequality
   \begin{equation*}
      \left|\Theta(x)\right|\leqslant\min\left(2a,\frac{1}{\pi|x|},\frac{1}{\pi|x|}\left(\frac{r}{2\pi|x|b}\right)^r\right).
   \end{equation*}
\end{lemma}
\begin{proof}
 See Piatetski--Shapiro~\cite{Piatetski-Shapiro-1952} or Segal~\cite{Segal-1933-1}.
\end{proof}

\begin{lemma}\label{Greave-result}
Suppose that $D>4$ is a real number and let $\lambda^{\pm}(d)$ be the Rosser's functions of level $D$. Then we have
the following properties. \\
\noindent
\emph{(1)} For any positive integer $d$, we have
\begin{equation*}
  |\lambda^{\pm}(d)|\leqslant1,\qquad \lambda^{\pm}(d)=0\quad \textrm{if \quad $d>D$ \quad or \quad$\mu(d)=0$ }.
\end{equation*}
\noindent
\emph{(2)} If $n$ is a positive integer, then
\begin{equation*}
  \sum_{d|n}\lambda^-(d)\leqslant\sum_{d|n}\mu(d)\leqslant\sum_{d|n}\lambda^+(d).
\end{equation*}
\noindent
\emph{(3)} If $z$ is a real number and if
\begin{equation}\label{M-def}
  P(z)=\prod_{2<p<z}p,\quad \mathfrak{P}=\prod_{2<p<z}\bigg(1-\frac{1}{p-1}\bigg),\quad
  \mathcal{M}^{\pm}=\sum_{d|P(z)}\frac{\lambda^\pm(d)}{\varphi(d)},\quad s_0=\frac{\log D}{\log z},
\end{equation}
then we have
\begin{align*}
  & \mathfrak{P}\leqslant\mathcal{M}^+\leqslant\mathfrak{P}\Big(F(s_0)+O\big((\log D)^{-1/3}\big)\Big),\\
  & \mathfrak{P}\geqslant\mathcal{M}^-\geqslant\mathfrak{P}\Big(f(s_0)+O\big((\log D)^{-1/3}\big)\Big),
\end{align*}
where $F(s)$ and $f(s)$ denote the classical functions in the linear sieve theory defined by
\begin{equation*}
  F(s)=
       \begin{cases}
         \displaystyle\frac{2e^\gamma}{s}, & \textrm{if}\,\,\, 1\leqslant s\leqslant3, \\
         \displaystyle\frac{2e^\gamma}{s}\bigg(1+\int_2^{s-1}\frac{\log(t-1)}{t}\mathrm{d}t\bigg),
                              & \textrm{if}\,\,\, 3\leqslant s\leqslant5, \\
       \end{cases}
\end{equation*}
and
\begin{equation*}
   f(s)=\frac{2e^\gamma\log(s-1)}{s}, \qquad \textrm{if}\,\,\, 2\leqslant s\leqslant 4.
\end{equation*}
Here $\gamma$ stands for the Euler's constant.
\end{lemma}
\begin{proof}
 This is a special case of a more general result. For the details, one can see Chapter 4 of Greaves \cite{Greaves-book}.
\end{proof}

\begin{lemma}\label{vector-inequality}
Let
\begin{equation*}
 \Lambda_i=\sum_{d|(p_i+2,P(z))}\mu(d),\qquad \Lambda_i^\pm=\sum_{d|(p_i+2,P(z))}\lambda^\pm(d),\qquad i=1,2,3,4.
\end{equation*}
Then we have
\begin{equation*}
 \Lambda_1\Lambda_2\Lambda_3\Lambda_4\geqslant\Lambda_1^-\Lambda_2^+\Lambda_3^+\Lambda_4^+ +\Lambda_1^+\Lambda_2^-\Lambda_3^+\Lambda_4^+  +\Lambda_1^+\Lambda_2^+\Lambda_3^-\Lambda_4^+
 +\Lambda_1^+\Lambda_2^+\Lambda_3^+\Lambda_4^-  -3\Lambda_1^+\Lambda_2^+\Lambda_3^+\Lambda_4^+ .
\end{equation*}
\end{lemma}
\begin{proof}
 See Lemma 13 of Br\"{u}dern and Fouvry \cite{Brudern-Fouvry-1994}.
\end{proof}

\begin{lemma}\label{Iwaniec-Kowalski-8.17}
  For any complex numbers $z_n$, we have
\begin{equation*}
  \Bigg|\sum_{a<n\leqslant b}z_n\Bigg|^2\leqslant\bigg(1+\frac{b-a}{Q}\bigg)\sum_{|q|<Q}\bigg(1-\frac{|q|}{Q}\bigg)
  \sum_{a<n,n+q\leqslant b}z_{n+q}\overline{z_n},
\end{equation*}
where $Q$ is any positive integer.
\end{lemma}
\begin{proof}
See Lemma 8.17 of Iwaniec and Kowalski \cite{Iwaniec-Kowalski-book-2004}.
\end{proof}

\begin{lemma}\label{expo-pair-gernal}
Suppose that $f(x):[a,b]\to\mathbb{R}$ has continuous derivatives of arbitrary order on $[a,b]$, where $1\leqslant a<b\leqslant2a$. Suppose further that
\begin{equation*}
 \big|f^{(j)}(x)\big|\asymp \lambda_1 a^{1-j},\qquad j\geqslant1, \qquad x\in[a,b].
\end{equation*}
Then for any exponential pair $(\kappa,\lambda)$, we have
\begin{equation*}
 \sum_{a<n\leqslant b}e(f(n))\ll \lambda_1^\kappa a^\lambda+\lambda_1^{-1}.
\end{equation*}
\end{lemma}
\begin{proof}
 See (3.3.4) of Graham and Kolesnik \cite{Graham-Kolesnik-book}.
\end{proof}

\begin{lemma}\label{T_upper}
For fixed constant $\mu\in(0,1)$, let
\begin{equation*}
 \mathcal{T}(x)=\sum_{d\leqslant D}\sum_{\substack{\mu X<n\leqslant X\\ d|n+2}}e(n^cx).
\end{equation*}
Then, for $0<|x|\leqslant2\Xi$, we have
\begin{equation*}
 \mathcal{T}(x)\ll X^{\frac{33c+11}{84}+\varepsilon}D^{\frac{73}{84}}+|x|^{-1}X^{1-c}\log X.
\end{equation*}
\end{lemma}
\begin{proof}
 Obviously, we have
\begin{align} \label{T_1}
          \mathcal{T}(x)
 = & \,\, \sum_{d\leqslant D}\sum_{(\mu X+2)/d<h\leqslant(X+2)/d}e\big((hd-2)^cx\big)
               \nonumber \\
 \ll & \,\,(\log X)\max_{\mathscr{D}\leqslant D}\sum_{\mathscr{D}<D\leqslant2\mathscr{D}}\Bigg|
           \sum_{(\mu X+2)/d<h\leqslant (X+2)/d}e\big((hd-2)^cx\big) \Bigg|.
\end{align}
For the inner sum in (\ref{T_1}), it follows from Lemma \ref{expo-pair-gernal} with exponential pair
\begin{equation*}
(\kappa,\lambda)=BA^3 BA^2 B(0,1)=\bigg(\frac{11}{28},\frac{11}{21}\bigg)
\end{equation*}
that
\begin{align} \label{T_2}
            \sum_{(\mu X+2)/d<h\leqslant (X+2)/d}e\big((hd-2)^cx\big)
 \ll & \,\, \big(|x|dX^{c-1}\big)^{\frac{11}{28}}\bigg(X/d\bigg)^{\frac{11}{21}}+\frac{X^{1-c}}{|x|d}
                   \nonumber \\
 \ll & \,\, |x|^\frac{11}{28}d^{-\frac{11}{84}}X^{\frac{33c+11}{84}}+\frac{X^{1-c}}{|x|d}.
\end{align}
Putting (\ref{T_2}) into (\ref{T_1}), we obtain
\begin{align*}
           \mathcal{T}(x)
 \ll & \,\,(\log X)\max_{\mathscr{D}\leqslant D}\sum_{\mathscr{D}<D\leqslant2\mathscr{D}}
           \bigg(
           |x|^\frac{11}{28}d^{-\frac{11}{84}}X^{\frac{33c+11}{84}}+\frac{X^{1-c}}{|x|d}\bigg)
                   \nonumber \\
 \ll & \,\, X^{\frac{33c+11}{84}+\varepsilon}D^{\frac{73}{84}}+|x|^{-1}X^{1-c}\log X,
\end{align*}
which completes the proof of Lemma \ref{T_upper}.
\end{proof}

For fixed constant $\mu\in(0,1)$, we define
\begin{equation}\label{I(x)-def}
  I(x)=\int_{\mu X}^Xe(t^cx)\mathrm{d}t.
\end{equation}
Moreover, we define
\begin{equation}\label{L(x)-def}
  L(x)=\sum_{d\leqslant D}\lambda(d)\sum_{\substack{\mu X<p\leqslant X\\ d|p+2}}(\log p)e(p^cx),
\end{equation}
where $\lambda(d)$ are real numbers satisfying
\begin{equation}\label{lambda-condi}
  |\lambda(d)|\leqslant1,\qquad \lambda(d)=0\quad \textrm{if \quad $2|d$ \quad or \quad$\mu(d)=0$ }.
\end{equation}

\begin{lemma}\label{L-I-substi}
Let $I(x)$ and $L(x)$ be defined as above.
Then, for $|x|\leqslant\tau$ and any sufficiently large $A>0$, we have
\begin{equation*}
 L(x)=\sum_{d\leqslant D}\frac{\lambda(d)}{\varphi(d)}I(x)+O\bigg(\frac{X}{(\log X)^A}\bigg),
\end{equation*}
provided that
\begin{equation*}
2\xi+16\delta<5 \qquad \textrm{and}\qquad 7\xi+14\delta<5.
\end{equation*}
\end{lemma}
\begin{proof}
See Lemma 2.12 of Li, Xue and Zhang \cite{Li-Xue-Zhang-2024}.
\end{proof}

\begin{lemma}\label{three-inte-upper}
  Let $I(x)$ and $L(x)$ be defined as above. Then we have
\begin{align*}
   &  \int_{|x|\leqslant\tau}\big|L(x)\big|^3\mathrm{d}x\ll X^{3-c}(\log X)^8,     \\
   &  \int_{|x|\leqslant\tau}\big|I(x)\big|^3\mathrm{d}x\ll X^{3-c}(\log X)^4,     \\
   &  \int_{|x|\leqslant\Xi}\big|L(x)\big|^2\mathrm{d}x\ll X\Xi(\log X)^8.
\end{align*}
\end{lemma}
\begin{proof}
By using the trivial estimate $|L(x)|\ll X\log^2X$ and $|I(x)|\ll X$, the conclusions follow from Lemma 11 of
Tolev \cite{Tolev-2017}.
\end{proof}

\begin{lemma}\label{Vaughan-decom}
  Let $f(n)$ be a complex--valued function defined for integers $n\in(\mu X,X]$. Then we have
\begin{equation*}
  \sum_{\mu X<n\leqslant X}\Lambda(n)f(n)=S_1-S_2-S_3,
\end{equation*}
where
\begin{align*}
   & S_1=\sum_{k\leqslant X^{1/3}}\mu(k)\sum_{\frac{\mu X}{k}<\ell\leqslant\frac{X}{k}}(\log\ell)f(k\ell),    \\
   & S_2=\sum_{k\leqslant X^{2/3}}c(k)\sum_{\frac{\mu X}{k}<\ell\leqslant\frac{X}{k}}f(k\ell),    \\
   & S_3=\sum_{X^{1/3}<k\leqslant X^{2/3}}a(k)\sum_{\frac{\mu X}{k}<\ell\leqslant\frac{X}{k}}\Lambda(\ell)f(k\ell),
\end{align*}
and where $a(k)$ and $c(k)$ are real numbers satisfying
\begin{equation*}
  \big|a(k)\big|\leqslant d(k),\qquad \qquad  \big|c(k)\big|\leqslant \log k.
\end{equation*}
\end{lemma}
\begin{proof}
See the arguments on page 112 of Vaughan \cite{Vaughan-1980}.      $\hfill$
\end{proof}

\begin{lemma}\label{Expo-esti}
Suppose that $1<c<\frac{5179}{3909}$. Suppose also that the real numbers $\lambda(d)$ satisfy (\ref{lambda-condi}) and $L(x)$
is defined by (\ref{L(x)-def}). Then we have
\begin{equation*}
  \sup_{\tau\leqslant|x|\leqslant\Xi}\big|L(x)\big|\ll X^{(695549-66179c)/640000-\varepsilon}.
\end{equation*}
\end{lemma}
\begin{proof}
Obviously, we have
\begin{equation*}
  L(x)=\sum_{d\leqslant D}\lambda(d)\sum_{\substack{\mu X<n\leqslant X\\ d|n+2}}\Lambda(n)e(n^cx)+O\big(X^{1/2+\varepsilon}\big)=:L_1(x)+O\big(X^{1/2+\varepsilon}\big).
\end{equation*}
Trivially, we only need to show that
\begin{equation*}
  \sup_{\tau\leqslant|x|\leqslant\Xi}\big|L_1(x)\big|\ll X^{(695549-66179c)/640000-\varepsilon}
\end{equation*}
holds for $1<c<\frac{5179}{3909}$. We rewrite $L_1(x)$ in the form
\begin{equation*}
  L_1(x)=\sum_{\mu X<n\leqslant X}\Lambda(n)f(n),
\end{equation*}
where
\begin{equation*}
 f(n)=\sum_{\substack{d\leqslant D\\ d|n+2}}\lambda(d)e(n^cx).
\end{equation*}
By Lemma \ref{Vaughan-decom}, we can see that
\begin{equation*}
 L_1(x)=S_1-S_2-S_3,
\end{equation*}
where
\begin{align*}
   & S_1=\sum_{k\leqslant X^{1/3}}\mu(k)\sum_{\mu X/k<\ell\leqslant X/k}(\log\ell)f(k\ell),    \\
   & S_2=\sum_{k\leqslant X^{2/3}}c(k)\sum_{\mu X/k<\ell\leqslant X/k}f(k\ell),    \\
   & S_3=\sum_{X^{1/3}<k\leqslant X^{2/3}}a(k)\sum_{\mu X/k<\ell\leqslant X/k}\Lambda(\ell)f(k\ell),
\end{align*}
and $|a(k)|\leqslant d(k),|c(k)|\leqslant\log k$. Moreover, we write $S_2=S_2^{(1)}+S_2^{(2)}$, where
\begin{equation*}
 S_2^{(1)}=\sum_{k\leqslant X^{1/3}}c(k)\sum_{\mu X/k<\ell\leqslant X/k}f(k\ell),\qquad
 S_2^{(2)}=\sum_{X^{1/3}<k\leqslant X^{2/3}}c(k)\sum_{\mu X/k<\ell\leqslant X/k}f(k\ell).
\end{equation*}
Therefore, we get
\begin{equation*}
 L_1(x)\ll \big|S_1\big|+\big|S_2^{(1)}\big|+\big|S_2^{(2)}\big|+\big|S_3\big|.
\end{equation*}
First, we consider the sum $S_2^{(1)}$. By noting the fact that $\lambda(d)=0$ for $2|d$, we can represent $S_2^{(1)}$ as
\begin{equation*}
 S_2^{(1)}=\sum_{\substack{d\leqslant D\\ (d,2)=1}}\lambda(d)\sum_{k\leqslant X^{1/3}}c(k)
  \sum_{\substack{\mu X/k<\ell\leqslant X/k\\ d|k\ell+2}}e\big((k\ell)^cx\big).
\end{equation*}
Since $(d,2)=1$ and $d|k\ell+2$, we have $(k,d)=1$, and thus $d|k\ell+2$ is equivalent to
$\ell\equiv\ell_0\pmod d$ for some
fixed integer $\ell_0$. This means that $\ell=\ell_0+md$ for some integer $m$. Hence, we obtain
\begin{equation}\label{S_2-re-write}
 S_2^{(1)}=\sum_{\substack{d\leqslant D\\ (d,2)=1}}\lambda(d)\sum_{\substack{k\leqslant X^{1/3}\\(k,d)=1}}c(k)
  \sum_{\mu X/(kd)-\ell_0/d<m\leqslant X/(kd)-\ell_0/d}e\big(k^c(\ell_0+md)^cx\big).
\end{equation}
Setting $h(m)=k^c(\ell_0+md)^cx$, then we have
\begin{equation*}
 \big|h^{(j)}(m)\big|\asymp\big(|x|dkX^{c-1}\big)\cdot\big(Xk^{-1}d^{-1}\big)^{1-j},\qquad j\geqslant1.
\end{equation*}
By Lemma \ref{expo-pair-gernal} with $(\kappa,\lambda)=A(\frac{1}{2},\frac{1}{2})=(\frac{1}{6},\frac{2}{3})$,
we know that the inner sum over $m$ in (\ref{S_2-re-write}) is
\begin{align}\label{S_2-inner}
  \ll & \,\, \big(|x|dkX^{c-1}\big)^{\frac{1}{6}}\big(Xk^{-1}d^{-1}\big)^{\frac{2}{3}}
             +\big(|x|dkX^{c-1}\big)^{-1}
                \nonumber \\
  \ll & \,\, |x|^\frac{1}{6}X^{\frac{c}{6}+\frac{1}{2}}k^{-\frac{1}{2}}d^{-\frac{1}{2}}
             +|x|^{-1}k^{-1}d^{-1}X^{1-c}.
\end{align}
Putting (\ref{S_2-inner}) into (\ref{S_2-re-write}), we get
\begin{equation}\label{S_2-upper}
  S_2^{(1)}\ll X^\varepsilon\big(D^{\frac{1}{2}}|x|^\frac{1}{6}X^{\frac{c}{6}+\frac{2}{3}}+|x|^{-1}X^{1-c}\big).
\end{equation}
For the sum $S_1$, by partial summation we can get rid of the factor $\log\ell$ and then proceed as the process of $S_2^{(1)}$ to derive that
\begin{equation}\label{S_1-upper}
  S_1\ll X^\varepsilon\big(D^{\frac{1}{2}}|x|^\frac{1}{6}X^{\frac{c}{6}+\frac{2}{3}}+|x|^{-1}X^{1-c}\big).
\end{equation}
Now, we consider the sum $S_3$. By a splitting argument, we can divide $S_3$ into $O(\log X)$ sums of the form
\begin{equation*}
  W(K):=\sum_{K<k\leqslant K_1}a(k)\sum_{\mu X/k<\ell\leqslant X/k}\Lambda(\ell)
  \sum_{\substack{d\leqslant D\\ d|k\ell+2}}\lambda(d)e\big((k\ell)^cx\big),
\end{equation*}
where
\begin{equation*}
  K<K_1\leqslant2K,\qquad  X^{1/3}\leqslant K<K_1\leqslant X^{2/3}.
\end{equation*}
Next, we shall consider the case $K\geqslant X^{1/2}$ firstly. Trivially, we have
\begin{equation*}
 W(K)\ll X^\varepsilon \sum_{K<k\leqslant K_1}
 \Bigg|\sum_{\mu X/k<\ell\leqslant X/k}\Phi(\ell)\Bigg|,
\end{equation*}
where
\begin{equation*}
\Phi(\ell)=\Lambda(\ell)\sum_{\substack{d\leqslant D\\ d|k\ell+2}}\lambda(d)e\big((k\ell)^cx\big).
\end{equation*}
It follows from Cauchy's inequality that
\begin{equation}\label{W(K)-Cauchy}
 |W(K)|^2\ll X^\varepsilon K\sum_{K<k\leqslant K_1}
 \Bigg|\sum_{\mu X/k<\ell\leqslant X/k}\Phi(\ell)\Bigg|^2.
\end{equation}
Suppose that $H$ is an integer which satisfies
\begin{equation*}
 1\leqslant H\ll \frac{X}{K}.
\end{equation*}
For the innermost sum on the right--hand side of (\ref{W(K)-Cauchy}), we use Lemma \ref{Iwaniec-Kowalski-8.17} to derive that
\begin{align*}
            \big|W(K)\big|^2
 \ll & \,\, \frac{X^{1+\varepsilon}}{H}\sum_{K<k\leqslant K_1}\sum_{|h|<H}\bigg(1-\frac{|h|}{H}\bigg)
            \sum_{\mu X/k<\ell,\ell+h\leqslant X/k}\Lambda(\ell)\Lambda(\ell+h)
                 \nonumber \\
 & \,\, \times\sum_{\substack{d_1\leqslant D\\ d_1|k\ell+2}}\lambda(d_1)e\big(-(k\ell)^cx\big)
         \sum_{\substack{d_2\leqslant D\\ d_2|k(\ell+h)+2}}\lambda(d_2)e\big((k(\ell+h))^cx\big)
                 \nonumber \\
 \ll & \,\, \frac{X^{1+\varepsilon}}{H}\sum_{d_1\leqslant D}\sum_{d_2\leqslant D}\lambda(d_1)\lambda(d_2)\sum_{|h|<H}
            \bigg(1-\frac{|h|}{H}\bigg)
                 \nonumber \\
 & \,\, \times\sum_{\mu X/K_1<\ell,\ell+h\leqslant X/K}\Lambda(\ell+h)\Lambda(\ell)
        \sum_{\substack{\widetilde{K}<k\leqslant\widetilde{K_1}\\ d_1|k\ell+2\\ d_2|k(\ell+h)+2}}
        e\Big(k^c\big((\ell+h)^c-\ell^c\big)x\Big),
\end{align*}
where
\begin{equation*}
  \widetilde{K}=\max\bigg(K,\frac{\mu X}{\ell},\frac{\mu X}{\ell+h}\bigg),\qquad
  \widetilde{K_1}=\min\bigg(K_1,\frac{X}{\ell},\frac{X}{\ell+h}\bigg).
\end{equation*}
By noting that $\lambda(d)=0$ for $2|d$, we can assume that $2\nmid d_1d_2$. Moreover, it follows from $d_1|k\ell+2$ and
$d_2|k(\ell+h)+2$ that $(d_1,\ell)=(d_2,\ell+h)=1$. Hence there exists some fixed integer $k_0=k_0(\ell,h,d_1,d_2)$ such
that the pair of conditions $d_1|k\ell+2$ and $d_2|k(\ell+h)+2$ is equivalent to $k\equiv k_0(\bmod {[d_1,d_2]})$. This means
that $k=k_0+m[d_1,d_2]$ for some integer $m$. Therefore, we get
\begin{align*}
 \mathscr{F}:= &\,\, \sum_{\substack{\widetilde{K}<k\leqslant\widetilde{K_1}\\ d_1|k\ell+2\\ d_2|k(\ell+h)+2}}
                      e\Big(k^c\big((\ell+h)^c-\ell^c\big)x\Big)
                          \nonumber \\
  = & \,\, \sum_{\frac{\widetilde{K}-k_0}{[d_1,d_2]}<m\leqslant\frac{\widetilde{K_1}-k_0}{[d_1,d_2]}}
            e\Big(\big(k_0+m[d_1,d_2]\big)^c\big((\ell+h)^c-\ell^c\big)x\Big)
                          \nonumber \\
  =: & \,\, \sum_{\frac{\widetilde{K}-k_0}{[d_1,d_2]}<m\leqslant\frac{\widetilde{K_1}-k_0}{[d_1,d_2]}}e(F(m)),
\end{align*}
say. Trivially, for $h=0$, we have $\mathscr{F}\ll K[d_1,d_2]^{-1}$. By the elementary estimate
\begin{equation*}
  \sum_{d_1\leqslant D}\sum_{d_2\leqslant D}[d_1,d_2]^{-1}\ll(\log D)^3,
\end{equation*}
we can see that the contribution of $\mathscr{F}$ with $h=0$ to $|W(K)|^2$ is $\ll X^{2+\varepsilon}H^{-1}$. For the case
$h\not=0$, we have
\begin{equation*}
  \big|F^{(j)}(m)\big|\asymp |x|h\ell^{c-1}[d_1,d_2]K^{c-1}\cdot \big(K[d_1,d_2]^{-1}\big)^{1-j}, \qquad j\geqslant1.
\end{equation*}
It follows from Lemma \ref{expo-pair-gernal} with
$(\kappa,\lambda)=BA BA^2 BA^2 B(0,1)=(\frac{11}{53},\frac{33}{53})$ that
\begin{align*}
       \mathscr{F}
\ll & \,\, |x|^{-1}h^{-1}\ell^{1-c}[d_1,d_2]^{-1}K^{1-c}+\big(|x|h\ell^{c-1}[d_1,d_2]K^{c-1}\big)^{\frac{11}{53}}
        \big(K[d_1,d_2]^{-1}\big)^{\frac{33}{53}}
               \nonumber \\
\ll & \,\, |x|^{-1}h^{-1}\ell^{1-c}[d_1,d_2]^{-1}K^{1-c}+|x|^{\frac{11}{53}}h^{\frac{11}{53}}\ell^{\frac{11(c-1)}{53}}
           [d_1,d_2]^{-\frac{22}{53}}K^{\frac{11c+22}{53}}.
\end{align*}
From the following estimate
\begin{align*}
            \sum_{d_1\leqslant D}\sum_{d_2\leqslant D}[d_1,d_2]^{-\frac{22}{53}}
 \ll & \,\, \sum_{d_1\leqslant D}\sum_{d_2\leqslant D}\bigg(\frac{(d_1,d_2)}{d_1d_2}\bigg)^{\frac{22}{53}}
            =\sum_{1\leqslant r\leqslant D}\sum_{k_1\leqslant\frac{D}{r}}\sum_{k_2\leqslant\frac{D}{r}}
            \frac{1}{r^{\frac{22}{53}}k_1^{\frac{22}{53}}k_2^{\frac{22}{53}}}
                      \nonumber \\
 \ll & \,\, \sum_{1\leqslant r\leqslant D}r^{-\frac{22}{53}}
            \Bigg(\sum_{k\leqslant \frac{D}{r}}k^{-\frac{22}{53}}\Bigg)^2
            \ll\sum_{1\leqslant r\leqslant D}r^{-\frac{22}{53}}\bigg(\frac{D}{r}\bigg)^{\frac{62}{53}}
            \ll D^{\frac{62}{53}},
\end{align*}
we can see that the contribution of $\mathscr{F}$ with $h\not=0$ to $|W(K)|^2$ is
\begin{align*}
\ll & \,\, X^{1+\varepsilon}H^{-1}\sum_{d_1\leqslant D}\sum_{d_2\leqslant D}\sum_{0<|h|<H}
           \sum_{\mu X/K_1<\ell,\ell+h\leqslant X/K}\Lambda(\ell+h)\Lambda(\ell)
                    \nonumber \\
 & \,\, \times \Big(|x|^{-1}h^{-1}\ell^{1-c}[d_1,d_2]^{-1}K^{1-c}+|x|^{\frac{11}{53}}h^{\frac{11}{53}}
               \ell^{\frac{11(c-1)}{53}}[d_1,d_2]^{-\frac{22}{53}}K^{\frac{11c+22}{53}}\Big)
                    \nonumber \\
\ll & \,\, X^{1+\varepsilon}H^{-1}|x|^{-1}K^{1-c}\Bigg(\sum_{d_1\leqslant D}
           \sum_{d_2\leqslant D}[d_1,d_2]^{-1}\Bigg)
           \Bigg(\sum_{0<|h|<H}h^{-1}\Bigg)
           \Bigg(\sum_{\frac{\mu X}{K_1}<\ell,\ell+h\leqslant\frac{X}{K}}\ell^{1-c}\Bigg)
                     \nonumber \\
& \,\, +X^{1+\varepsilon}H^{-1}|x|^{\frac{11}{53}}K^{\frac{11c+22}{53}}
        \Bigg(\sum_{d_1\leqslant D}\sum_{d_2\leqslant D}[d_1,d_2]^{-\frac{22}{53}}\Bigg)
                    \nonumber \\
 & \,\,   \qquad \times \Bigg(\sum_{0<|h|<H}h^{\frac{11}{53}}\Bigg)
          \Bigg(\sum_{\mu X/K_1<\ell,\ell+h\leqslant X/K}\ell^{\frac{11(c-1)}{53}}\Bigg)
                    \nonumber \\
\ll & \,\, X^{3-c+\varepsilon}|x|^{-1}H^{-1}K^{-1}+X^{\frac{11c+95}{53}+\varepsilon}H^{\frac{11}{53}}D^{\frac{62}{53}}
           |x|^{\frac{11}{53}}K^{-\frac{20}{53}}.
\end{align*}
Combining the above two cases, we obtain
\begin{equation}\label{W(K)-upper-1}
  \big|W(K)\big|^2\ll X^\varepsilon\Big(X^2H^{-1}+X^{3-c}|x|^{-1}H^{-1}K^{-1}+X^{(11c+95)/53}H^{11/53}D^{62/53}|x|^{11/53}K^{-20/53}\Big).
\end{equation}
Taking
\begin{equation*}
  H_0=X^{(11-11c)/64}D^{-31/32}|x|^{-11/64}K^{5/16},
		\qquad H=[\min(H_0,XK^{-1})],
\end{equation*}
it is easy to see that
\begin{equation}\label{H-rela}
  H^{-1}\asymp H_0^{-1}+KX^{-1}.
\end{equation}
By using (\ref{W(K)-upper-1}) and (\ref{H-rela}), and noting that $K\geqslant X^{1/2}$, we derive that
\begin{align*}
           \big|W(K)\big|^2
\ll & \,\, X^\varepsilon\Big(X^2\big(H_0^{-1}+KX^{-1}\big)+X^{3-c}|x|^{-1}K^{-1}\big(H_0^{-1}+KX^{-1}\big)
                 \nonumber \\
 & \,\,\qquad +X^{(11c+95)/53}H_0^{11/53}D^{62/53}|x|^{11/53}K^{-20/53}\Big)
                 \nonumber \\
\ll & \,\, X^\varepsilon\Big(X^{(11c+107)/64}D^{31/32}|x|^{11/64}
			+X^{5/3}+X^{(139-53c)/64}D^{31/32}|x|^{-53/64}
			+X^{2-c}|x|^{-1}\Big),
\end{align*}
which implies that
\begin{align}\label{W(K)-upp-fi}
 \big|W(K)\big|
 \ll & \,\, X^\varepsilon\Big(X^{(11c+107)/128}D^{31/64}|x|^{11/128}
			+X^{5/6}+X^{(139-53c)/128}D^{31/64}|x|^{-53/128}
                \nonumber \\
	 & \,\, \qquad +X^{1-c/2}|x|^{-1/2}\Big).
\end{align}
If $K<X^{1/2}$, we write $W(K)$ as
\begin{equation*}
  W(K)=\sum_{\mu X/K_1<\ell\leqslant X/K}\Lambda(\ell)
   \sum_{\max(K,\mu X/\ell)<k\leqslant\min(K_1,X/\ell)}a(k)\sum_{\substack{d\leqslant D\\ d|k\ell+2}}\lambda(d)e\big((k\ell)^cx\big).
\end{equation*}
Then we have $X/K\gg X^{1/2}$ and we can proceed as the previous process by changing the roles of the variables $k$
and $\ell$ reversed. Therefore, we also derive the estimate (\ref{W(K)-upp-fi}) in this case. Consequently, we obtain
\begin{align}\label{S_3-upper}
S_3
 \ll & \,\, X^\varepsilon\Big(X^{(11c+107)/128}D^{31/64}|x|^{11/128}
			+X^{5/6}+X^{(139-53c)/128}D^{31/64}|x|^{-53/128}
                \nonumber \\
	 & \,\, \qquad +X^{1-c/2}|x|^{-1/2}\Big).
\end{align}
For $S_2^{(2)}$, we can use the same way to give the upper bound estimate by the expression of the right--hand side of (\ref{S_3-upper}). Above all, from (\ref{S_2-upper}), (\ref{S_1-upper}) and (\ref{S_3-upper}), we deduce that
\begin{align*}
            L_1(x)
 \ll & \,\, X^\varepsilon\Big(X^{c/6+2/3}D^{1/2}|x|^{1/6}
			+X^{(11c+107)/128}D^{31/64}|x|^{11/128}+X^{5/6}
                    \nonumber \\
	 & \,\, \qquad+X^{(139-53c)/128}D^{31/64}|x|^{-53/128}+X^{1-c/2}|x|^{-1/2}+X^{1-c}|x|^{-1}\Big)
                    \nonumber \\
 \ll & \,\, X^\varepsilon\Big(X^{c/6+2/3+\delta/2}+X^{(11c+107)/128+31\delta/64}
            +X^{5/6}+X^{139/128+31\delta/64-53\xi/128} +X^{1-\xi/2}\Big),
\end{align*}
and thus, for $1<c<\frac{5179}{3909}$, there holds
\begin{equation*}
   \sup_{\tau\leqslant|x|\leqslant\Xi}\big|L_1(x)\big|\ll X^{(695549-66179c)/640000-\varepsilon}.
\end{equation*}
This completes the proof Lemma \ref{Expo-esti}.     $\hfill$
\end{proof}

\begin{lemma}\label{inte-minor}
  For $1<c<\frac{5179}{3909}$, we have
\begin{equation*}
   \int_{\tau\leqslant|x|\leqslant\Xi}\big|L(x)\big|^4\mathrm{d}x\ll X^{4-c-\varepsilon}.
\end{equation*}
\end{lemma}
\begin{proof}
Suppose that $\mathscr{H}(x)$ denotes a integrable function constrained on $[-\Xi,\Xi]$. Then we have
\begin{align*}
   \bigg|\int_{\tau\leqslant|x|\leqslant\Xi}L(x)\mathscr{H}(x)\mathrm{d}x\bigg|
= & \,\, \Bigg|\sum_{d\leqslant D}\lambda(d)\sum_{\substack{\mu X<p\leqslant X\\ d|p+2}}(\log p)
         \int_{\tau\leqslant|x|\leqslant\Xi}e(p^cx)\mathscr{H}(x)\mathrm{d}x\Bigg|
                \nonumber \\
\leqslant & \,\,(\log X)\sum_{d\leqslant D}\sum_{\substack{\mu X<p\leqslant X\\ d|p+2}}
                \Bigg|\int_{\tau\leqslant|x|\leqslant\Xi}e(p^cx)\mathscr{H}(x)\mathrm{d}x\Bigg|
                 \nonumber \\
\leqslant & \,\,(\log X)\sum_{d\leqslant D}\sum_{\substack{\mu X<n\leqslant X\\ d|n+2}}
                \Bigg|\int_{\tau\leqslant|x|\leqslant\Xi}e(n^cx)\mathscr{H}(x)\mathrm{d}x\Bigg|.
\end{align*}
By Cauchy's inequality, we have
\begin{align}\label{inte-1}
      & \,\,     \bigg|\int_{\tau\leqslant|x|\leqslant\Xi}L(x)\mathscr{H}(x)\mathrm{d}x\bigg|^2
\ll  (\log X)^2\Bigg(\sum_{\mu X<n\leqslant X}\sum_{\substack{d\leqslant D\\ d|n+2}}
           \bigg|\int_{\tau\leqslant|x|\leqslant\Xi}e(n^cx)\mathscr{H}(x)\mathrm{d}x\bigg|\Bigg)^2
                  \nonumber \\
\ll & \,\, X(\log X)^2\sum_{\mu X<n\leqslant X}\Bigg(\sum_{\substack{d\leqslant D\\ d|n+2}}
           \bigg|\int_{\tau\leqslant|x|\leqslant\Xi}e(n^cx)\mathscr{H}(x)\mathrm{d}x\bigg|\Bigg)^2
                  \nonumber \\
\ll & \,\, X(\log X)^2 \sum_{\mu X<n\leqslant X}\Bigg(\sum_{\substack{d\leqslant D\\ d|n+2}}1\Bigg)
           \sum_{\substack{d\leqslant D\\ d|n+2}}
           \bigg|\int_{\tau\leqslant|x|\leqslant\Xi}e(n^cx)\mathscr{H}(x)\mathrm{d}x\bigg|^2
                  \nonumber \\
\ll & \,\, X^{1+\frac{\varepsilon}{4}}\sum_{d\leqslant D}\sum_{\substack{\mu X<n\leqslant X\\ d|n+2}}
           \bigg|\int_{\tau\leqslant|x|\leqslant\Xi}e(n^cx)\mathscr{H}(x)\mathrm{d}x\bigg|^2
                  \nonumber \\
= & \,\, X^{1+\frac{\varepsilon}{4}}\sum_{d\leqslant D}\sum_{\substack{\mu X<n\leqslant X\\ d|n+2}}
         \bigg(\int_{\tau\leqslant|x|\leqslant\Xi}e(n^cx)\mathscr{H}(x)\mathrm{d}x\bigg)
         \bigg(\int_{\tau\leqslant|y|\leqslant\Xi}\overline{e(n^cy)\mathscr{H}(y)}\mathrm{d}y\bigg)
                  \nonumber \\
\ll & \,\, X^{1+\frac{\varepsilon}{4}}\int_{\tau\leqslant|y|\leqslant\Xi}\big|\mathscr{H}(y)\big|\mathrm{d}y
           \int_{\tau\leqslant|x|\leqslant\Xi}\big|\mathscr{H}(x)\big|\big|\mathcal{T}(x-y)\big|\mathrm{d}x.
\end{align}
For the innermost integral on the right--hand side of (\ref{inte-1}), we have
\begin{align}\label{inner-inte-gene}
      & \,\, \int_{\tau\leqslant|x|\leqslant\Xi}\big|\mathscr{H}(x)\big|\big|\mathcal{T}(x-y)\big|\mathrm{d}x
                 \nonumber \\
 \ll  & \,\, \int_{\substack{\tau\leqslant|x|\leqslant\Xi\\ |x-y|\leqslant X^{-c}}}
             \big|\mathscr{H}(x)\big|\big|\mathcal{T}(x-y)\big|\mathrm{d}x
             +\int_{\substack{\tau\leqslant|x|\leqslant\Xi\\ X^{-c}<|x-y|\leqslant 2\Xi}}
             \big|\mathscr{H}(x)\big|\big|\mathcal{T}(x-y)\big|\mathrm{d}x.
\end{align}
By the trivial estimate $|\mathcal{T}(x-y)|\ll X\log X$, we have
\begin{align}\label{inner-inte-1}
           \int_{\substack{\tau\leqslant|x|\leqslant\Xi\\ |x-y|\leqslant X^{-c}}}
           \big|\mathscr{H}(x)\big|\big|\mathcal{T}(x-y)\big|\mathrm{d}x
\ll & \,\, X(\log X) \times\sup_{\tau\leqslant|x|\leqslant\Xi}\big|\mathscr{H}(x)\big|
           \times\int_{\substack{\tau\leqslant|x|\leqslant\Xi\\ |x-y|\leqslant X^{-c}}}\mathrm{d}x
                   \nonumber \\
\ll & \,\, X^{1-c}(\log X)\times\sup_{\tau\leqslant|x|\leqslant\Xi}\big|\mathscr{H}(x)\big|.
\end{align}
It follows from Lemma \ref{T_upper} that
\begin{align}\label{inner-inte-2}
 & \,\, \int_{\substack{\tau\leqslant|x|\leqslant\Xi\\ X^{-c}<|x-y|\leqslant 2\Xi}}
        \big|\mathscr{H}(x)\big|\big|\mathcal{T}(x-y)\big|\mathrm{d}x
                   \nonumber \\
\ll & \,\, \int_{\substack{\tau\leqslant|x|\leqslant\Xi\\ X^{-c}<|x-y|\leqslant 2\Xi}}\big|\mathscr{H}(x)\big|
           \bigg(X^{\frac{33c+11}{84}+\varepsilon}D^{\frac{73}{84}}+|x-y|^{-1}X^{1-c}\log X\bigg)
           \mathrm{d}x
                   \nonumber \\
\ll & \,\, X^{\frac{33c+11}{84}+\varepsilon}D^{\frac{73}{84}}\int_{\tau\leqslant|x|\leqslant\Xi}
           \big|\mathscr{H}(x)\big|\mathrm{d}x
           +X^{1-c+\varepsilon}\times\sup_{\tau\leqslant|x|\leqslant\Xi}\big|\mathscr{H}(x)\big|\times
           \int_{\substack{\tau\leqslant|x|\leqslant\Xi\\ X^{-c}<|x-y|\leqslant 2\Xi}}\frac{\mathrm{d}x}{|x-y|}
                    \nonumber \\
\ll & \,\, X^{\frac{33c+11}{84}+\varepsilon}D^{\frac{73}{84}}\int_{\tau\leqslant|x|\leqslant\Xi}
           \big|\mathscr{H}(x)\big|\mathrm{d}x
           +X^{1-c+\varepsilon}\times\sup_{\tau\leqslant|x|\leqslant\Xi}\big|\mathscr{H}(x)\big|.
\end{align}
Combining (\ref{inte-1}), (\ref{inner-inte-gene}), (\ref{inner-inte-1}) and (\ref{inner-inte-2}), we derive that
\begin{align}\label{inte-ite}
\bigg|\int_{\tau\leqslant |x|\leqslant \Xi}L(x)\mathscr{H}(x)\mathrm{d}x\bigg|^2
\ll & \,\, X^{2-c+\varepsilon}\sup_{\tau\leqslant|x|\leqslant\Xi}|\mathscr{H}(x)|
			\int_{\tau\leqslant|x|\leqslant\Xi}|\mathscr{H}(x)|\mathrm{d}x
                 \nonumber \\
	& \,\,\quad+X^{(33c+95)/84+\varepsilon}D^{73/84}\left(\int_{\tau\leqslant |x|\leqslant \Xi}
          |\mathscr{H}(x)|\mathrm{d}x\right)^2.
\end{align}
Taking $\mathscr{H}(x)=\overline{L(x)}|L(x)|$ in (\ref{inte-ite}), it follows from Lemma \ref{three-inte-upper}
and Lemma \ref{Expo-esti} that
\begin{align}\label{eq:5.36}
			\int_{\tau\leqslant |x|\leqslant \Xi}|L(x)|^3\mathrm{d}x
 \ll & \,\, X^{1-c/2+(695549-66179c)/640000+\varepsilon}
				\left(\int_{\tau\leqslant |x|\leqslant\Xi}|L(x)|^2\mathrm{d}x\right)^{1/2}
				 \nonumber \\
	 & \,\, \quad+X^{(33c+95)/168+\varepsilon}D^{73/168}\int_{\tau\leqslant |x|\leqslant \Xi}|L(x)|^2\mathrm{d}x
                  \nonumber \\
 \ll & \,\, X^{3/2-c/2+(695549-66179c)/640000+\varepsilon}+X^{(33c+263)/168+\varepsilon}D^{73/168}.
\end{align}
Taking $\mathscr{H}(x)=\overline{L(x)}|L(x)|^2$ in (\ref{inte-ite}), it follows from (\ref{eq:5.36}) that
\begin{align}\label{eq:5.37}
           \int_{\tau\leqslant |x|\leqslant \Xi}|L(x)|^4\,dx
\ll & \,\, X^\varepsilon\Big(X^{7/4-3c/4+(695549-66179c)/320000}+X^{2+(33c+11)/84}D^{73/84}
				\nonumber \\
	& \,\, \qquad +X^{7/4-c/2+3(695549-66179c)/1280000+(33c+11)/336}D^{73/336}
				\nonumber \\
	& \,\, \qquad+X^{2-c/2+(695549-66179c)/640000+(33c+11)/168}D^{73/168}\Big).
\end{align}
Accordingly, we deduce that
\begin{equation*}
\int_{\tau\leqslant|x|\leqslant\Xi}\big|L(x)\big|^4\mathrm{d}x
  \ll X^{4-c-\varepsilon},
\end{equation*}
provided that
\begin{equation*}
1<c<\frac{5179}{3909}.
\end{equation*}
This completes the proof of Lemma \ref{inte-minor}.
\end{proof}

\section{Proof of Theorem \ref{Theorem-1}}
Consider the sum
\begin{equation}\label{sum-Thm}
 \Gamma:=\sum_{\substack{\mu X<p_1,p_2,p_3,p_4\leqslant X\\ |p_1^c+p_2^c+p_3^c+p_4^c-N|<(\log N)^{-E}\\ (p_i+2,P(z))=1\\ i=1,2,3,4}}
        (\log p_1)(\log p_2)(\log p_3)(\log p_4).
\end{equation}
In order to prove Theorem \ref{Theorem-1}, we only need to show that $\Gamma>0$. Suppose that $\Theta(x)$ and $\vartheta(y)$
are the functions which are defined in Lemma \ref{xiaobei-lemma} with parameters
$a=\frac{7}{8}(\log N)^{-E},b=\frac{1}{8}(\log N)^{-E}$ and $r=[\log^2X]$. Therefore, we get
\begin{equation}\label{xiaobei-condi}
  \vartheta(y)=0\quad \textrm{if}\quad |y|\geqslant(\log N)^{-E};\quad 0<\vartheta(y)<1\quad
  \textrm{if}\quad |y|<(\log N)^{-E}.
\end{equation}
Obviously, it follows from (\ref{sum-Thm}) and (\ref{xiaobei-condi}) that
\begin{equation}\label{Gamma-lower-1}
 \Gamma\geqslant\widetilde{\Gamma}:=\sum_{\substack{\mu X<p_1,p_2,p_3,p_4\leqslant X\\
  (p_i+2,P(z))=1\\ i=1,2,3,4}}
 (\log p_1)(\log p_2)(\log p_3)(\log p_4)\vartheta(p_1^c+p_2^c+p_3^c+p_4^c-N).
\end{equation}
By the definition of $\Lambda_i$ in Lemma \ref{vector-inequality}, we know that
\begin{equation*}
 \Lambda_i=\sum_{d|(p_i+2,P(z))}\mu(d)=
 \begin{cases}
   1, & \textrm{if $(p_i+2,P(z))=1$},  \\
   0, & \textrm{otherwise},
 \end{cases}
\end{equation*}
which combined with Lemma \ref{vector-inequality} yields that
\begin{align*}
           \widetilde{\Gamma}
  = & \,\, \sum_{\mu X<p_1,p_2,p_3,p_4\leqslant X}(\log p_1)(\log p_2)(\log p_3)
           (\log p_4)\Lambda_1\Lambda_2\Lambda_3\Lambda_4
           \vartheta(p_1^c+p_2^c+p_3^c+p_4^c-N)
                 \nonumber \\
\geqslant & \,\, \sum_{\mu X<p_1,p_2,p_3,p_4\leqslant X}(\log p_1)(\log p_2)
                 (\log p_3)(\log p_4)\vartheta(p_1^c+p_2^c+p_3^c+p_4^c-N)
                 \nonumber \\
    & \,\, \times\big(\Lambda_1^-\Lambda_2^+\Lambda_3^+\Lambda_4^++\Lambda_1^+\Lambda_2^-\Lambda_3^+\Lambda_4^+
            +\Lambda_1^+\Lambda_2^+\Lambda_3^-\Lambda_4^++\Lambda_1^+\Lambda_2^+\Lambda_3^+\Lambda_4^-
            -3\Lambda_1^+\Lambda_2^+\Lambda_3^+\Lambda_4^+\big)
                 \nonumber \\
  = & \,\, \Gamma_1+\Gamma_2+\Gamma_3+\Gamma_4-3\Gamma_5,
\end{align*}
say. Trivially, by the symmetric property, we can see that
\begin{align*}
 & \Gamma_1=\Gamma_2=\Gamma_3=\Gamma_4=\sum_{\mu X<p_1,p_2,p_3,p_4\leqslant X}
   \Bigg(\prod_{j=1}^4\log p_j\Bigg)
            \Lambda_1^-\Lambda_2^+\Lambda_3^+\Lambda_4^+\vartheta(p_1^c+p_2^c+p_3^c+p_4^c-N),
                \nonumber \\
 & \Gamma_5=\sum_{\mu X<p_1,p_2,p_3,p_4\leqslant X}\Bigg(\prod_{j=1}^4\log p_j\Bigg)
            \Lambda_1^+\Lambda_2^+\Lambda_3^+\Lambda_4^+\vartheta(p_1^c+p_2^c+p_3^c+p_4^c-N),
\end{align*}
and thus
\begin{equation}\label{Gamma-tidle-lower-1}
  \widetilde{\Gamma}\geqslant4\Gamma_1-3\Gamma_5.
\end{equation}
Let $\lambda^\pm(d)$ be the Rosser's functions of level $D$, and define
\begin{align*}
  L^\pm(x):= & \,\, \sum_{\mu X<p\leqslant X}(\log p)e(p^cx)\sum_{d|(p+2,P(z))}\lambda^\pm(d)
               \nonumber \\
           = & \,\, \sum_{d|P(z)}\lambda^\pm(d) \sum_{\substack{\mu X<p\leqslant X\\ d|p+2}}(\log p)e(p^cx).
\end{align*}
According to the Fourier's inverse transformation, we have
\begin{align}\label{Gamma_1-decom}
 \Gamma_1 = & \,\,\sum_{\mu X<p_1,p_2,p_3,p_4\leqslant X}(\log p_1)(\log p_2)(\log p_3)(\log p_4)
 \Lambda_1^-\Lambda_2^+\Lambda_3^+\Lambda_4^+
                \nonumber \\
 & \,\, \qquad\times \int_{-\infty}^{+\infty}\Theta(x)e\big((p_1^c+p_2^c+p_3^c+p_4^c-N)x\big)\mathrm{d}x
                \nonumber \\
 = & \,\, \int_{-\infty}^{+\infty}L^-(x)(L^+(x))^3\Theta(x)e(-Nx)\mathrm{d}x
                \nonumber \\
 = & \,\, \bigg(\int_{|x|\leqslant\tau}+\int_{\tau<|x|<\Xi}+\int_{|x|\geqslant\Xi}\bigg)
           L^-(x)(L^+(x))^3\Theta(x)e(-Nx)\mathrm{d}x
                \nonumber \\
 = & \,\, \Gamma_1^{(1)}+\Gamma_1^{(2)}+\Gamma_1^{(3)},
\end{align}
say. Similarly, for $\Gamma_5$, we also have
\begin{align}\label{Gamma_5-decom}
          \Gamma_5
 = & \,\, \int_{-\infty}^{+\infty}(L^+(x))^4\Theta(x)e(-Nx)\mathrm{d}x
                \nonumber \\
 = & \,\, \bigg(\int_{|x|\leqslant\tau}+\int_{\tau<|x|<\Xi}+\int_{|x|\geqslant\Xi}\bigg)(L^+(x))^4\Theta(x)e(-Nx)\mathrm{d}x
                \nonumber \\
 =: & \,\, \Gamma_5^{(1)}+\Gamma_5^{(2)}+\Gamma_5^{(3)}.
\end{align}
By the trivial estimate $L^\pm(x)\ll X^{1+\varepsilon}$ and Lemma \ref{xiaobei-lemma}, we get
\begin{align}\label{Gamma_1,5-(3)}
             \Gamma_1^{(3)},\Gamma_5^{(3)}
  \ll & \,\, X^{4+\varepsilon}\int_{\Xi}^\infty\frac{1}{\pi|x|}\bigg(\frac{r}{2\pi|x|b}\bigg)^r\mathrm{d}x
                 \nonumber \\
  \ll & \,\, X^{4+\varepsilon}\bigg(\frac{r}{2\pi b}\bigg)^r\int_{\Xi}^\infty\frac{\mathrm{d}x}{x^{r+1}}
             \ll X^{4+\varepsilon}\bigg(\frac{r}{2\pi\Xi b}\bigg)^r
                 \nonumber \\
  \ll & \,\, \frac{X^{4+\varepsilon}}{(2\pi\log X)^{\log X}}\ll
             \frac{X^{4+\varepsilon}}{X^{\log\log X+\log(2\pi)}}\ll1.
\end{align}
It follows from (\ref{Gamma-lower-1}), (\ref{Gamma-tidle-lower-1}), (\ref{Gamma_1-decom}), (\ref{Gamma_5-decom}) and (\ref{Gamma_1,5-(3)}) that
\begin{equation}\label{Gamma-lower-decom}
  \Gamma\geqslant\Big(4\Gamma_1^{(1)}-3\Gamma_5^{(1)}\Big)+\Big(4\Gamma_1^{(2)}-3\Gamma_5^{(2)}\Big)+O(1).
\end{equation}
By Lemma \ref{L-I-substi}, we know that, for $|x|\leqslant\tau$, there holds
\begin{equation*}
  L^\pm(x)=\mathcal{M}^\pm I(x)+O\bigg(\frac{X}{(\log X)^A}\bigg),
\end{equation*}
where $\mathcal{M}^\pm$ and $I(x)$ are defined by (\ref{M-def}) and (\ref{I(x)-def}), respectively. By noting the identity
\begin{align*}
   & \,\, L^-(x)\big(L^+(x)\big)^3-\mathcal{M}^-\big(\mathcal{M}^+\big)^3I^4(x)
                 \nonumber \\
 = & \,\, \big(L^+(x)\big)^3\big(L^-(x)-\mathcal{M}^-I(x)\big)+\mathcal{M}^-I(x)\big(L^+(x)\big)^2
          \big(L^+(x)-\mathcal{M}^+I(x)\big)
                 \nonumber \\
   & \,\,  +\mathcal{M}^-\mathcal{M}^+I^2(x)L^+(x)\big(L^+(x)-\mathcal{M}^+I(x)\big)
                +\mathcal{M}^-\big(\mathcal{M}^+\big)^2I^3(x)\big(L^+(x)-\mathcal{M}^+I(x)\big)
\end{align*}
and the elementary estimate
\begin{equation*}
  \mathcal{M}^\pm\ll\sum_{d\leqslant D}\frac{1}{\varphi(d)}\ll\log X,
\end{equation*}
we derive that
\begin{align*}
  & \,\, \Big|L^-(x)\big(L^+(x)\big)^3-\mathcal{M}^-\big(\mathcal{M}^+\big)^3I^4(x)\Big|
                  \nonumber \\
\ll & \,\, \frac{X}{(\log X)^A}\Big(\big|L^+(x)\big|^3+(\log X)\big|L^+(x)\big|^2\big|I(x)\big|
           +(\log X)^2\big|L^+(x)\big|\big|I(x)\big|^2+(\log X)^3\big|I(x)\big|^3\Big)
                  \nonumber \\
\ll & \,\, \frac{X}{(\log X)^{A-3}}\Big(\big|L^+(x)\big|^3+\big|I(x)\big|^3\Big).
\end{align*}
Define
\begin{equation*}
  \mathcal{I}_0=\int_{|x|\leqslant\tau}I^4(x)\Theta(x)e(-Nx)\mathrm{d}x.
\end{equation*}
Then from Lemma \ref{xiaobei-lemma} and Lemma \ref{three-inte-upper}, we derive that
\begin{align}\label{inte-err}
                 \Big|\Gamma_1^{(1)}-\mathcal{M}^-\big(\mathcal{M}^+\big)^3\mathcal{I}_0\Big|
\leqslant & \,\, \int_{|x|\leqslant\tau}\Big|L^-(x)\big(L^+(x)\big)^3
                 -\mathcal{M}^-\big(\mathcal{M}^+\big)^3I^4(x)\Big|\big|\Theta(x)\big|\mathrm{d}x
                        \nonumber \\
\ll & \,\, \frac{X}{(\log X)^{A+E-3}}\Bigg(\int_{|x|\leqslant\tau}\big|L^+(x)\big|^3\mathrm{d}x
           +\int_{|x|\leqslant\tau}\big|I(x)\big|^3\mathrm{d}x\Bigg)
                        \nonumber \\
\ll & \,\, \frac{X^{4-c}}{(\log X)^{A+E-11}}.
\end{align}
Set
\begin{equation*}
  \mathcal{I}=\int_{-\infty}^{+\infty}I^4(x)\Theta(x)e(-Nx)\mathrm{d}x.
\end{equation*}
It follows from Lemma 8 of Zhai and Cao \cite{Zhai-Cao-2003} that
\begin{equation}\label{singular-I-lower}
  \mathcal{I}\gg\frac{X^{4-c}}{(\log X)^{E}}.
\end{equation}
Since $\big|\frac{\mathrm{d}}{\mathrm{d}t}(t^cx)\big|\gg|x|X^{c-1}$ for $t\in(\mu X,X]$, by Lemma 4.2 of Titchmarsh
\cite{Titchmarsh-book}, we have $|I(x)|\ll (|x|X^{c-1})^{-1}$. Therefore, we obtain
\begin{align}\label{I-I_0-error}
     \big|\mathcal{I}-\mathcal{I}_0\big|
\ll & \,\, \int_{\tau}^{\infty}\big|I(x)\big|^4\big|\Theta(x)\big|\mathrm{d}x\ll(\log X)^{-E}\int_{\tau}^{\infty}
           \bigg(\frac{1}{|x|X^{c-1}}\bigg)^4\mathrm{d}x
                \nonumber \\
\ll & \,\, (\log X)^{-E}X^{4-4c}\tau^{-3}\ll\frac{X^{4-c-3\xi}}{(\log X)^E}.
\end{align}
Combining (\ref{inte-err}) and (\ref{I-I_0-error}), we get
\begin{align}\label{main-substi-1}
          \Gamma_1^{(1)}
 = & \,\, \mathcal{M}^-\big(\mathcal{M}^+\big)^3\mathcal{I}_0+O\bigg(\frac{X^{4-c}}{(\log X)^{A+E-11}}\bigg)
                  \nonumber \\
 = & \,\, \mathcal{M}^-\big(\mathcal{M}^+\big)^3\bigg(\mathcal{I}
          +O\bigg(\frac{X^{4-c-3\xi}}{(\log X)^E}\bigg)\bigg)
          +O\bigg(\frac{X^{4-c}}{(\log X)^{A+E-11}}\bigg)
                  \nonumber \\
 = & \,\, \mathcal{M}^-\big(\mathcal{M}^+\big)^3\mathcal{I}+O\bigg(\frac{X^{4-c}}{(\log X)^{A+E-11}}\bigg).
\end{align}
Similarly, we can get
\begin{equation}\label{main-substi-2}
\Gamma_5^{(1)}=\big(\mathcal{M}^+\big)^4\mathcal{I}+O\bigg(\frac{X^{4-c}}{(\log X)^{A+E-11}}\bigg).
\end{equation}
From Mertens' prime number theorem (See \cite{Mertens-1874}), we know that
\begin{equation*}
   \mathfrak{P}\asymp\frac{1}{\log X},
\end{equation*}
which combined with (\ref{singular-I-lower}), (\ref{main-substi-1}), (\ref{main-substi-2}) and Lemma \ref{Greave-result} yields
\begin{align}\label{main-term-lower}
      & \,\, 4\Gamma_1^{(1)}-3\Gamma_5^{(1)}
             =\big(4\mathcal{M}^--3\mathcal{M}^+\big)\big(\mathcal{M}^+\big)^3\mathcal{I}
             +O\bigg(\frac{X^{4-c}}{(\log X)^{A+E-11}}\bigg)
                       \nonumber \\
\geqslant & \,\,\bigg(4f\bigg(\frac{\log D}{\log z}\bigg)
                -3F\bigg(\frac{\log D}{\log z}\bigg)\bigg)\big(1+O(\log^{-1/3}X)\big)
                \mathfrak{P}^4\mathcal{I}+O\bigg(\frac{X^{4-c}}{(\log X)^{A+E-11}}\bigg)
                       \nonumber \\
= & \,\, \bigg(4f\bigg(\frac{25}{8}\bigg)-3F\bigg(\frac{25}{8}\bigg)\bigg)\mathfrak{P}^4\mathcal{I}
         +O\bigg(\frac{X^{4-c}}{(\log X)^{E+13/3}}\bigg)
                        \nonumber \\
= & \,\, \frac{64e^\gamma}{25}\Bigg(\log\frac{17}{8}-\frac{3}{4}
         \bigg(1+\int_2^{\frac{17}{8}}\frac{\log(t-1)}{t}\mathrm{d}t\bigg)\Bigg)
         \mathfrak{P}^4\mathcal{I}+O\bigg(\frac{X^{4-c}}{(\log X)^{E+13/3}}\bigg)
                        \nonumber \\
  \geqslant & \,\, 0.0048722\times\mathfrak{P}^4\mathcal{I}+O\bigg(\frac{X^{4-c}}{(\log X)^{E+13/3}}\bigg).
\end{align}
For $\Gamma_1^{(2)}$, by H\"{o}lder's inequality, Lemma \ref{xiaobei-lemma} and Lemma \ref{inte-minor}, we obtain
\begin{align}\label{Gamma_1-(2)-upper}
           \big|\Gamma_1^{(2)}\big|
\ll & \,\, \int_{\tau\leqslant|x|\leqslant\Xi}\big|L^-(x)\big|\big|L^+(x)\big|^3\big|\Theta(x)\big|\mathrm{d}x
                     \nonumber \\
\ll & \,\, (\log X)^{-E}\int_{\tau\leqslant|x|\leqslant\Xi}\big|L^-(x)\big|\big|L^+(x)\big|^3\mathrm{d}x
                     \nonumber \\
\ll & \,\, (\log X)^{-E}\bigg(\int_{\tau\leqslant|x|\leqslant\Xi}\big|L^-(x)\big|^4\mathrm{d}x\bigg)^{\frac{1}{4}}
           \bigg(\int_{\tau\leqslant|x|\leqslant\Xi}\big|L^+(x)\big|^4\mathrm{d}x\bigg)^{\frac{3}{4}}
                     \nonumber \\
\ll & \,\, (\log X)^{-E}\cdot X^{4-c-\varepsilon}\ll X^{4-c-\varepsilon}.
\end{align}
Similarly, we have
\begin{equation}\label{Gamma_5-(2)-upper}
\big|\Gamma_5^{(2)}\big|\ll (\log X)^{-E}\int_{\tau\leqslant|x|\leqslant\Xi}\big|L^+(x)\big|^4\mathrm{d}x
\ll X^{4-c-\varepsilon}.
\end{equation}
According to (\ref{Gamma-lower-decom}), (\ref{main-term-lower}), (\ref{Gamma_1-(2)-upper}) and (\ref{Gamma_5-(2)-upper}), we deduce that
\begin{align*}
            \Gamma
\geqslant & \,\, \Big(4\Gamma_1^{(1)}-3\Gamma_5^{(1)}\Big)+O\Big(\Big|\Gamma_1^{(2)}\Big|+\Big|\Gamma_5^{(2)}\Big|+1\Big)
                    \nonumber \\
\geqslant & \,\, 0.0048722\times\mathfrak{P}^4\mathcal{I}
                 +O\bigg(\frac{X^{4-c}}{(\log X)^{E+13/3}}\bigg)+O\big(X^{4-c-\varepsilon}\big)
                    \nonumber \\
\gg & \,\, \frac{X^{4-c}}{(\log X)^{E+4}}.
\end{align*}
Therefore, $\Gamma>0$ for sufficiently large real number $N$. Then inequality (\ref{Thm-ineq}) would have a solution in primes $p_1,p_2,p_3,p_4$ satisfying
\begin{equation}\label{l-condition}
(p_1+2,P(z))=(p_2+2,P(z))=(p_3+2,P(z))=(p_4+2,P(z))=1.
\end{equation}
If the number $p_i+2$ has $l$ prime factors counted according to multiplicity, then from (\ref{l-condition}) and from the condition
$\mu X<p_i\leqslant X$, it is easy to find that $l\leqslant\eta^{-1}$, which means that $p_i+2$ would be almost--prime of order
$[\eta^{-1}]=[\frac{31250}{5179-3909c}]$.

This completes the proof of Theorem \ref{Theorem-1}.

\section*{Acknowledgements}

The authors would like to appreciate the referee for his/her patience in refereeing this paper.
This work is supported by Beijing Natural Science Foundation (Grant No. 1242003), and the
National Natural Science Foundation of China (Grant Nos. 12471009, 12301006, 11901566, 12001047).


\begin{thebibliography}{99}
\bibitem{Baker-Weingartner-2013}R. Baker, A. Weingartner, \textit{Some applications of the double large sieve},
                                        Monatsh. Math.,  \textbf{170} (2013), no. 3--4, 261--304.

\bibitem{Baker-Weingartner-2014}R. Baker, A. Weingartner, \textit{A ternary Diophantine inequality over primes},
                                        Acta Arith., \textbf{162} (2014), no. 2, 159--196.

\bibitem{Baker-2021}R. Baker, \textit{Some Diophantine equations and inequalities with primes},
                             Funct. Approx. Comment. Math., \textbf{64} (2021), no. 2, 203--250.

\bibitem{Brudern-Fouvry-1994}J. Br\"{u}dern, \'{E}. Fouvry, \textit{Lagrange's four squares theorem with almost
                               prime variables}, J. Reine Angew. Math., \textbf{454} (1994), 59--96.

\bibitem{Cai-1996}Y. C. Cai, \textit{A Diophantine inequality with prime variables},
                              Acta Math. Sinica (Chin. Ser.), \textbf{39} (1996), no. 6, 733--742.

\bibitem{Cai-1999}Y. C. Cai, \textit{On a Diophantine inequality involving prime numbers III},
                              Acta Math. Sin. (Engl. Ser.), \textbf{15} (1999), no. 3, 387--394.

\bibitem{Cai-2018}Y. C. Cai, \textit{A ternary Diophantine inequality involving primes},
                              Int. J. Number Theory, \textbf{14} (2018), no. 8, 2257--2268.

\bibitem{Cao-Zhai-2002}X. D. Cao, W. G. Zhai, \textit{A Diophantine inequality with prime numbers},
                              Acta Math. Sin. (Chin. Ser.), \textbf{45} (2002), no. 2, 361--370.

\bibitem{Chen-1966}J. R. Chen, \textit{On the representation of a large even integer as the sum of a prime and
                               the product of at most two primes}, Kexue Tongbao, \textbf{17} (1966), 385--386.

\bibitem{Chen-1973}J. R. Chen, \textit{On the representation of a larger even integer as the sum of a prime and
                               the product of at most two primes}, Sci. Sinica, \textbf{16} (1973), 157--176.

\bibitem{Dimitrov-2017-1}S. I. Dimitrov, \textit{On a diophantine inequality with prime powers of a special type},
                              Proc. Techn. Univ.-Sofia, \textbf{67} (2017), no. 3, 25--33.

\bibitem{Dimitrov-2017-2}S. I. Dimitrov, \textit{A ternary diophantine inequality over special primes},
             JP Journal of Algebra, Number Theory and Applications, \textbf{39} (2017), no. 3, 335--368.

\bibitem{Dimitrov-2017-3}S. I. Dimitrov, \textit{A quaternary Diophantine inequality by prime numbers of
                          a special type}, Proceedings of the Technical University of Sofia,
                          \textbf{67} (2017), no. 2, 317--326.

\bibitem{Garaev-2003}M. Z. Garaev, \textit{On the Waring--Goldbach problem with small non--integer exponent},
                                      Acta Arith., \textbf{108} (2003), no. 3, 297--302.

\bibitem{Graham-Kolesnik-book}S. W. Graham, G. Kolesnik, \textit{Van der Corput's method of exponential sums},
                             Cambridge University Press, New York, 1991.

\bibitem{Greaves-book}G. Greaves, \textit{Sieves in number theory}, Springer--Verlag, Berlin, 2001.

\bibitem{Hua-1938}L. K. Hua, \textit{Some results in the additive prime--number theory},
                                    Quart. J. Math. Oxford Ser. (2), \textbf{9} (1938), no. 1, 68--80.

\bibitem{Iwaniec-Kowalski-book-2004}H. Iwaniec, E. Kowalski, \textit{Analytic number theory},
                                    American Mathematical Society, Providence, RI, 2004.

\bibitem{Kumchev-Nedeva-1998}A. Kumchev, T. Nedeva, \textit{On an equation with prime numbers},
                                Acta Arith., \textbf{83} (1998), no. 2, 117--126.

\bibitem{Kumchev-1999}A. Kumchev, \textit{A Diophantine inequality involving prime powers},
                                Acta Arith., \textbf{89} (1999), no. 4, 311--330.

\bibitem{Li-Xue-Zhang-2022}J. Li, F. Xue, M. Zhang, \textit{A ternary Diophantine inequality with prime numbers of
                              a special form},  Period. Math. Hungar., \textbf{85} (2022), no. 1, 14--31.

\bibitem{Li-Xue-Zhang-2024}J. Li, F. Xue, M. Zhang, \textit{A quaternary Diophantine inequality with prime numbers of a special form}, Ramanujan J., \textbf{63} (2024), no. 2, 259--291.

\bibitem{Li-Cai-2020}S. H. Li, Y. C. Cai, \textit{On a Diophantine inequality involving prime numbers},
                               Ramanujan J., \textbf{52} (2020), no. 1, 163--174.

\bibitem{Matomaki-Shao-2017}K. Matom\"{a}ki, X. Shao, \textit{Vinogradov's three primes theorem with almost
                               twin primes}, Compos. Math., \textbf{153} (2017), no. 6, 1220--1256.

\bibitem{Mertens-1874}F. Mertens, \textit{Ein Beitrag zur analytyischen Zahlentheorie},
                                     J. Reine Angew. Math., \textbf{78} (1874), 46--62.

\bibitem{Mu-2015}Q. W. Mu, \textit{On a Diophantine inequality over primes},
                                    Adv. Math. (China), \textbf{44} (2015), no. 4, 621--637.

\bibitem{Piatetski-Shapiro-1952}I. I. Piatetski--Shapiro, \textit{On a variant of Waring--Goldbach's problem},
                                       Mat. Sb., \textbf{30(72)} (1952), no. 1, 105--120.

\bibitem{Segal-1933-1}B. I. Segal, \textit{On a theorem analogous to Waring's theorem},
                                   Dokl. Akad. Nauk SSSR (N. S.), \textbf{2} (1933), 47--49.

\bibitem{Shi-Liu-2013}S. Y. Shi, L. Liu, \textit{On a Diophantine inequality involving prime powers},
                                  Monatsh. Math., \textbf{169} (2013), no. 3--4, 423--440.

\bibitem{Titchmarsh-book}E. G. Titchmarsh, \textit{The Theory of the Riemann Zeta--function},
                               Oxford University Press, New York, 1986.

\bibitem{Tolev-thesis}D. I. Tolev, \textit{Diophantine inequalities involving prime numbers},
                                Ph.D. thesis, Moscow University, 1990.

\bibitem{Tolev-1992}D. I. Tolev, \textit{On a Diophantine inequality involving prime numbers},
                                  Acta Arith., \textbf{61} (1992), no. 3, 289--306.

\bibitem{Tolev-1999}D. I. Tolev, \textit{Arithmetic progressions of prime--almost--prime twins},
                                  Acta Arith., \textbf{88} (1999), no. 1, 67--98.

\bibitem{Tolev-2000-1}D. I. Tolev, \textit{Representations of large integers as sums of two primes of special
  type}, in: Algebraic number theory and Diophantine analysis (Graz, 1998), 485--495, de Gruyter, Berlin, 2000.

\bibitem{Tolev-2000-2}D. I. Tolev, \textit{Additive problems with prime numbers of special type},
                                 Acta Arith., \textbf{96} (2000), no. 1, 53--88.

\bibitem{Tolev-2017}D. I. Tolev, \textit{On a Diophantine inequality with prime numbers of a special type},
                                 Proc. Steklov Inst. Math., \textbf{299} (2017), no. 1, 246--267.

\bibitem{Vaughan-1980}R. C. Vaughan, \textit{An elementary method in prime number theory},
                                  Acta Arith., \textbf{37} (1980), 111--115.

\bibitem{Vinogradov-1937}I. M. Vinogradov, \textit{Representation of an odd number as the sum of three primes},
                                   Dokl. Akad. Nauk. SSSR, \textbf{15} (1937), 169--172.

\bibitem{Zhai-Cao-2003}W. G. Zhai, X. D. Cao, \textit{On a Diophantine inequality over primes},
                                    Adv. Math. (China), \textbf{32} (2003), no. 1, 63--73.

\bibitem{Zhai-Cao-2007}W. G. Zhai, X. D. Cao, \textit{On a Diophantine inequality over primes (II)},
                                    Monatsh. Math., \textbf{150} (2007), no. 2, 173--179.

\bibitem{Zhang-Li-2019-1}M. Zhang, J. Li, \textit{On a Diophantine inequality over primes},
                           J. Number Theory, \textbf{202} (2019), 220--253.

\bibitem{Zhang-Li-2019-2}M. Zhang, J. Li, \textit{A Diophantine inequality with four prime variables},
                           Int. J. Number Theory, \textbf{15} (2019), no. 9, 1759--1770.

\bibitem{Zhu-2020}L. Zhu, \textit{A ternary Diophantine inequality with prime numbers of a special type},
                          Proc. Indian Acad. Sci. Math. Sci., \textbf{130} (2020), no. 1, Paper No. 23, 14 pp.

\end{thebibliography}
\end{document}